\documentclass[11pt]{amsart}
\usepackage{amsmath}
\usepackage{a4wide}
\usepackage[utf8]{inputenc}
\usepackage{amssymb}
\usepackage{amsopn}
\usepackage{epsfig}
\usepackage{amsfonts}
\usepackage{latexsym}
\usepackage{amsthm}
\usepackage{enumerate}
\usepackage[UKenglish]{babel}
\usepackage{verbatim}
\usepackage{color}
\usepackage{pgf}
\usepackage{tikz}
\usetikzlibrary{positioning} 
\usepackage{xcolor} 
\usepackage{mathrsfs}   
\usepackage{bm}

 \usetikzlibrary{arrows,automata}

\DeclareMathAlphabet{\pazocal}{OMS}{zplm}{m}{n}
\newtheorem{theorem}{Theorem}[section]
\newtheorem{lemma}[theorem]{Lemma}
\newtheorem{proposition}[theorem]{Proposition}
\newtheorem{corollary}[theorem]{Corollary}

\theoremstyle{definition}

\theoremstyle{remark}
\newtheorem{remark}[theorem]{Remark}

\numberwithin{equation}{section}

\newcommand{\R}{\mathbb{R}}
\newcommand{\Q}{\mathbb{Q}}
\newcommand{\Z}{\mathbb{Z}}

\newcommand{\bbm}{\begin{bmatrix}}
\newcommand{\ebm}{\end{bmatrix}}

 \newcommand{\set}[1]{\left\{#1\right\}}
\newcommand{\N}{\ensuremath{\mathbb{N}}}
\newcommand{\f}{\infty}

\begin{document}
\title[rational points in Cantor sets and spectral eigenvalue problem]{Rational points in Cantor sets and spectral eigenvalue problem for self-similar spectral measures}

\author{Derong Kong}
\address[D. Kong]{College of Mathematics and Statistics, Center of Mathematics, Chongqing University, Chongqing, 401331, P.R.China.}
\email{derongkong@126.com}

\author{Kun Li}
\address[K. Li]{College of Mathematics and Statistics, Chongqing University, Chongqing, 401331, P.R.China}
\email{kunli@cqu.edu.cn}

\author{Zhiqiang Wang}
\address[Z. Wang]{College of Mathematics and Statistics, Center of Mathematics, Key Laboratory of Nonlinear Analysis and its Applications (Ministry of Education), Chongqing University, Chongqing 401331, P.R.China}
\email{zhiqiangwzy@163.com}



\begin{abstract}
Given $q\in \mathbb{N}_{\ge 3}$ and a finite set $A\subset\mathbb{Q}$, let
$$K(q,A)= \bigg\{\sum_{i=1}^{\infty} \frac{a_i}{q^{i}}:a_i \in A ~\forall i\in \mathbb{N} \bigg\}.$$
For $p\in\mathbb{N}_{\ge 2}$ let $D_p\subset\mathbb{R}$ be the set of  all rational numbers having a finite $p$-ary expansion. We show in this paper that for $p \in \mathbb{N}_{\ge 2}$ with $\gcd(p,q)=1$, the intersection $D_p\cap K(q, A)$ is a finite set if and only if $\dim_H K(q, A)<1$, which is also equivalent to the fact that the set $K(q, A)$ has no interiors.
We apply this result to study the spectral eigenvalue problem.
For a Borel probability measure $\mu$ on $\mathbb{R}$, a real number $t\in \mathbb{R}$ is called a spectral eigenvalue of $\mu$ if both $E(\Lambda) =\big\{ e^{2 \pi \mathrm{i} \lambda x}: \lambda \in \Lambda \big\}$ and $E(t\Lambda) = \big\{ e^{2 \pi \mathrm{i} t\lambda x}: \lambda \in \Lambda \big\}$ are orthonormal bases in $L^2(\mu)$ for some $\Lambda \subset \mathbb{R}$.
For any self-similar spectral measure generated by a Hadamard triple, we provide a class of spectral eigenvalues which is dense in $[0,+\infty)$, and show that every eigen-subspace associated with these spectral eigenvalues is infinite.
\end{abstract}

\maketitle

\section{Introduction}

\subsection{Rational points in Cantor sets}

Let $\N:=\set{1,2,3,\ldots}$ be the set of all positive integers, and let $\N_0:=\N\cup\set{0}$.  For $k\in\N$ let $\N_{\ge k}:=\N\cap[k,+\f)$.  For $n,m \in\Z$ let $\mathrm{gcd}(n, m)$ be their \emph{greatest common divisor}.
Let $\#A$ denote the cardinality of a set $A$ and let $\dim_H K$ be the Hausdorff dimension of a set $K$.

Given $q\in \N_{\ge 2}$ and a finite set $A \subset \mathbb{Q}$, let $K(q, A)$ be the self-similar set in $\mathbb R$ generated by the \emph{iterated function system} (IFS) (cf.~\cite{Hutchinson-1981})
\begin{equation}\label{eq:IFS-q-A}
{\mathcal F}_{q,A}=\bigg\{ f_a(x)=\frac{x+a}{q}: a\in A \bigg\}.\end{equation}
 Then the set $K(q, A)$ can be written algebraically as
\begin{equation}\label{eq:self-similar-set-K}
  K(q,A)=\bigg\{\sum_{i=1}^{\infty} \frac{a_i}{q^{i}}:a_i \in A ~~\forall i\in \N\bigg\}.
\end{equation}
To avoid the trivial situation, we always assume $\# A \ge 2$.
When $q=2$, the set $K(q,A)$ is always an interval. In the following we focus on the case $q \in \N_{\ge 3}$.

For $p\in \N_{\ge 2}$, let $D_p$ be the set of all rational numbers in $\mathbb {R}$ having a finite $p$-ary expansion. Then
\begin{equation}\label{Dp-definition}
D_p=\bigcup_{n=1}^\f \frac{\mathbb {Z}}{p^n},
\end{equation}
which is a proper subset of $\mathbb Q$ and  is dense in $\mathbb R$.

In 1984, Mahler \cite{Mahler-1984} raised the problem of how well irrational elements of the middle-third Cantor set can be approximated by rational numbers within it, and rational numbers outside of it.
A related question is the structure of rational points in Cantor sets \cite{Wall-1990,Nagy-2001,Bloshchitsyn-2015,Schleischitz-2021,Shparlinski-2021,Li-Li-Wu-2023,JKLW-2024}.
When $A \subset \{0,1,\ldots ,q-1 \}$ and $\#A<q$, if $\gcd(p,q)=1$, Schleischitz \cite[Corollary 4.4]{Schleischitz-2021} showed $\#\big( D_{p}\cap K(q,A) \big)<+\f$, and furthermore, Jiang et al.~\cite[Theorem 1.2]{JKLW-2024} proved $\#\big( (D_p+\alpha)\cap K(q, A) \big)< +\f$ for any $\alpha\in\R$.
In this paper we consider general digit sets $A\subset\mathbb Q$, and give necessary and sufficient conditions for $\#\big( D_p\cap K(q,A) \big)<+\f$.

\begin{theorem}\label{thm:equivalence-of-finite}
Let $q\in \mathbb{N}_{\ge 3}$ and let $A \subset \mathbb{Q}$ be a finite set. The  following statements are equivalent.
\begin{itemize}
  \item[{\rm(i)}] $\dim_{H} K(q,A)<1$;
  \item[{\rm(ii)}] the set $K(q,A)$ has no interiors;
  \item[{\rm(iii)}] $\#\big(D_p\cap K(q,A)\big)<+\f$ for any $p\in\mathbb N_{\ge 2}$ with $\gcd(p,q)=1$.
\end{itemize}
\end{theorem}
\begin{remark}\label{remark:intersection}
\begin{enumerate}[{\rm(i)}]
  \item The statement (iii) can be strengthened as $\#\big((rD_p+\alpha)\cap K(q,A)\big)<+\f$ for $r \in \Q \setminus \{0\}$, $\alpha \in \Q$, and $p\in\mathbb N_{\ge 2}$ with $\gcd(p,q)=1$. This is because $\#\big((rD_p+\alpha)\cap K(q,A)\big) = \# \big( D_p \cap K(q,\widetilde{A}) \big)$ where $\widetilde{A}=r^{-1}A - r^{-1}\alpha(q-1)$.

  \item Note that the IFS ${\mathcal F}_{q, A}$ defined in (\ref{eq:IFS-q-A}) satisfies the weak separation property (see Lemma \ref{IFS-is-WSP}). It follows from \cite[Theorem 3]{Zerner1996} that (ii) implies (i). Since $D_p$ is dense in $\R$, we obviously have (iii) implies (ii). It remains to prove that (i) implies (iii).

  \item When $\#A =q$, the set $K(q,A)$ is called a \emph{self-similar tile} if it has positive Lebesgue measure, see \cite[Theorem 1.1]{Lagarias-Wang-1996} for conditions characterizing self-similar tiles. For $q \in \N_{\ge 3}$ and $A\subset \Q$ with $\#A=q$, it follows from Theorem \ref{thm:equivalence-of-finite} that $K(q,A)$ is a self-similar tile if and only if there exists $p \in \N_{\ge 2}$ with $\gcd(p,q)=1$ such that $\#\big( D_p \cap K(q,A) \big) = +\f$.
\end{enumerate}
\end{remark}

For $q \in \N_{\ge 3}$ and $A =\{0,1,\ldots,q-2,q\}$, noting that $K(q,A) = K(q^2, qA +A)$ and $\#(qA+A) < q^2$, we have $\dim_H K(q,A) < 1$. 
By Theorem \ref{thm:equivalence-of-finite} we obtain the following corollary directly.

\begin{corollary}
  For $p \in \N_{\ge 2}$ and $q \in \N_{\ge 3}$ with $\gcd(p,q)=1$, we have $$\#\big( D_p \cap K(q,\{0,1,\ldots, q-2, q\}) \big) < + \f.$$
\end{corollary}

As another application of Theorem \ref{thm:equivalence-of-finite}, we partially extend \cite[Theorem 1.2]{JKLW-2024} and derive a uniform upper bound.

\begin{theorem}\label{thm:uniform-bound}
Let $q\in \mathbb{N}_{\ge 3}$ and $A \subset \mathbb{Q}$ with $\#(A-A) < q$. If $p\in \N_{\ge 2}$ with $\gcd(p,q)=1$, then we have
$$\sup_{\alpha \in \R}\#\big( (D_p+\alpha) \cap K(q,A) \big) <+\infty.$$
\end{theorem}

Let $\lfloor x \rfloor$ denote the biggest integer not larger than $x$. If $A \subset \Q$ with $\# A \le \sqrt{q}$ or $A \subset \big\{ 0,1, \ldots, \lfloor \frac{q}{2}\rfloor -1 \}$ for $q \in \N_{\ge 4}$, then by Theorem \ref{thm:uniform-bound} we have $$\sup_{\alpha \in \R}\#\big( (D_p+\alpha) \cap K(q,A) \big) <+\infty$$ for any $p\in \N_{\ge 2}$ with $\gcd(p,q)=1$.

\subsection{Spectral eigenvalue problem for spectral measures}
A Borel probability measure $\mu$ on $\R$ is called a \emph{spectral measure} if there exists $\Lambda \subset \R$ such that $E(\Lambda) = \big\{ e_\lambda(x):=e^{2\pi \mathrm{i} \lambda x}: \lambda \in \Lambda \big\}$ forms an orthonormal basis in $L^2(\mu)$. The set $\Lambda$ is called a \emph{spectrum} of $\mu$.
The study of spectral measures originated from the spectral set conjecture raised by Fuglede \cite{Fuglede-1974}.
In 1998, Jorgensen and Pedersen \cite{Jorgensen-Pedersen-1998} discovered a class of singular continuous self-similar spectral measures.
Since then, a large number of singular continuous spectral measures have been constructed, see \cite{Strichartz-2000,Laba-Wang-2002,An-He-2014,An-Fu-Lai-2019,Dutkay-Haussermann-Lai-2019,
Li-Miao-Wang-2022,Lu-Dong-Zhang-2022} and the references therein.

In 2002, {\L}aba and Wang \cite{Laba-Wang-2002} found the scaling spectrum phenomenon: both $\Lambda$ and $2\Lambda$ are spectra of a singular continuous spectral measure.
This amazing phenomenon cannot occur in absolutely continuous spectral measures.
A real number $t \in \R$ is called a \emph{spectral eigenvalue} of $\mu$ if both $\Lambda$ and $t\Lambda$ are spectra of $\mu$ for some $\Lambda \subset \R$.
The following \emph{spectral eigenvalue problem} naturally arises.
\begin{quote}
  \emph{To find all spectral eigenvalues of a singular continuous spectral measure.}
\end{quote}
Fu, He, and Wen \cite{Fu-He-Wen-2018} addressed the spectral eigenvalue problem for Bernoulli convolutions, and showed that for spectral Bernoulli convolutions $\mu_{2k,\{0,1\}}$ with $k \in \N_{\ge 2}$, the real number $t \in \R$ is a spectral eigenvalue of $\mu_{2k,\{0,1\}}$ if and only if $t$ is the quotient of two odd integers. This result was extended by Fu and He \cite{Fu-He-2017} to self-similar spectral measures with consecutive digits and some Moran spectral measures.
Spectral eigenvalues of self-similar spectral measures with product-form digit set were also characterized in \cite{Li-Wu-2022,Jiang-Lu-Wei-2024}.
In higher dimensions, the spectral eigenmatrix problem for some self-affine measures has been investigated as well \cite{An-Dong-He-2022,Chen-Liu-2023,Liu-Tang-Wu-2023,Liu-Liu-Tang-Wu-2024}.
On the other hand, for a canonical spectrum $\Lambda$ of a spectral measure $\mu$, to find all real number $t\in \R$ such that $t\Lambda$ is also a spectrum of $\mu$ is the other type of spectral eigenvalue problem, see \cite{Jorgensen-Kornelson-Shuman-2011,Dutkay-Jorgensen-2012,Dai-2016,Dutkay-Haussermann-2016,
He-Tang-Wu-2019,Li-Wu-2022,Jiang-Lu-Wei-2024} and the references therein.
In this paper, we give a class of spectral eigenvalues for self-similar spectral measures in $\R$.

Let $N \in \N_{\ge 2}$ and let $B,L \subset \Z$ be finite sets with $\# B = \#L \ge 2$.
We say $(N,B,L)$ is a \emph{Hadamard triple} in $\R$ if the matrix
\begin{equation}\label{eq:matrix}
  M=\bigg( \frac{1}{\sqrt{\# B}} e^{2\pi \mathrm{i} \frac{b\ell}{N}} \bigg)_{b \in B, \ell \in L}
\end{equation}
is unitary, i.e., $M^* M = I$.
This implies all elements in $B$ are distinct modula $N$. Thus, we have $\# B \le N$.
{Let $\mu_{N,B}$ denote the self-similar measure associated with the IFS $\mathcal{F}_{N,B} = \big\{ \tau_b(x) = (x+b)/N: b \in B \big\}$ (cf.~\cite{Hutchinson-1981}), i.e., $\mu_{N,B}$ is the unique Borel probability measure satisfying} $$\mu = \frac{1}{\#B} \sum_{b \in B} \mu\circ\tau_b^{-1}.$$
For a Hadamard triple $(N,B,L)$ in $\R$, {\L}aba and Wang \cite{Laba-Wang-2002} proved that the self-similar measure $\mu_{N,B}$ is a spectral measure, see \cite{Dutkay-Haussermann-Lai-2019} for generalization in higher dimensions. If $\#B =N$, then the self-similar spectral measure $\mu_{N,B}$ is absolutely continuous.
We focus on the case $\# B < N$.

\begin{theorem}\label{thm:spectral-eigenvalue}
  Let $(N,B,L)$ be a Hadamard triple in $\R$ with $\#B < N$.
  For any $p_1,p_2,\cdots,p_k \in \N_{\ge 2}$ with $\gcd(p_1p_2\ldots p_k, N) =1$, there exists $\Lambda \subset \R$ such that
  \[ p_1^{n_1} p_2^{n_2} \ldots p_k^{n_k} \Lambda \text{ is a spectrum of the self-similar measure $\mu_{N,B}$} \]
  for all tuples $(n_1,n_2,\ldots, n_k) \in \N_0^k$.
\end{theorem}

As a direct corollary, we obtain a class of spectral eigenvalues for general self-similar spectral measures, and moreover, {every eigen-subspace associated with these spectral eigenvalues is infinite.}

\begin{corollary}\label{cor:eigenvalue}
  Let $(N,B,L)$ be a Hadamard triple in $\R$ with $\#B < N$.
  Write \[ \mathcal{E}_N:= \bigg\{ \frac{v}{u}:\; u,v \in \N,\; \gcd(u,v)=1,\; \gcd(uv,N)=1  \bigg\}. \]
  Then any rational number $t \in \mathcal{E}_N$ is a spectral eigenvalue of $\mu_{N,B}$.
  Moreover, for any $t_1,t_2,\ldots,t_k \in \mathcal{E}_N$, the eigen-subspace \[ \big\{ \Lambda \subset \R:\; 0 \in \Lambda\;\text{ and }\; \Lambda, t_1\Lambda, t_2 \Lambda,\ldots,t_k \Lambda \text{ are spectra of $\mu_{N,B}$} \big\} \]
  is infinite.
\end{corollary}

Note that the set $\mathcal{E}_N$ is dense in $[0,+\f)$.
In general, the set $\mathcal{E}_N$ only provides some spectral eigenvalues of self-similar spectral measure.
For self-similar spectral measures studied in \cite{Li-Wu-2022}, the set $\mathcal{E}_N$ offers all spectral eigenvalues, and the sufficiency part of \cite[Theorem 1.9]{Li-Wu-2022} follows directly from Corollary \ref{cor:eigenvalue}.
In terms of eigen-subspace, there exist some spectral eigenvalues of spectral Bernoulli convolutions, whose eigen-subspace is of continuum cardinality \cite{Fu-He-Wen-2018}.
By Corollary \ref{cor:eigenvalue}, we have known that any eigen-subspace of spectral eigenvalues in $\mathcal{E}_N$ is infinite.
It is of interest to consider the following question.

\noindent
\textbf{Question:} Let $(N,B,L)$ be a Hadamard triple in $\R$ with $\#B < N$.
For any $t \in \mathcal{E}_N$, is the eigen-subspace $ \big\{ \Lambda \subset \R:\; 0 \in \Lambda\;\text{ and }\; \Lambda, t\Lambda \text{ are spectra of } \mu_{N,B} \big\}$ of continuum cardinality?

The paper is organized as follows: in section \ref{sec:intersection-Q}, we focus on rational points in Cantor sets, and prove {Theorems \ref{thm:equivalence-of-finite} and \ref{thm:uniform-bound}}; we study spectral eigenvalues of self-similar spectral measures in section \ref{sec:spectral-eigenvalue}, and prove Theorem \ref{thm:spectral-eigenvalue}.

\section{Finiteness of the intersection  $D_p\cap K(q, A)$ for $A\subset\mathbb Q$}\label{sec:intersection-Q}

Given $p\in \N_{\ge 2}, q \in \N_{\ge 3}$ with $\gcd(p,q)=1$ and a finite set $A\subset \mathbb Q$, we show in this section that if $\dim_H K(q, A) < 1$ then the intersection $D_p\cap K(q, A)$ is finite, and then we prove Theorem \ref{thm:uniform-bound}.

For a finite set $\Sigma$, let $\Sigma^* := \{ \mathbf{i}=i_1 i_2 \ldots i_n: n \in \N, i_1, i_2, \ldots i_n \in \Sigma \}$ be the set of all finite words over the alphabet $\Sigma$.
Given a self-similar IFS $\mathcal G=\big\{ g_i(x)=r_i x+b_i \big\}_{i=1}^N$ in $\mathbb R$, for $\mathbf{i}=i_1 i_2 \ldots i_n \in \{1,2,\ldots,N\}^*$, define $g_{\mathbf{i}}:= g_{i_1} \circ g_{i_2} \circ \ldots \circ g_{i_n}$.
The IFS $\mathcal{G}$ is said to satisfy the \emph{weak separation property} (WSP) (cf. Lau and Ngai \cite{Lau-Ngai-1999} and Zerner \cite{Zerner1996}) if the identity map is not an accumulation point of \[ \mathcal N_{\mathcal G}:=\set{g_{\mathbf i}^{-1}\circ g_{\mathbf j}: {\mathbf i, \mathbf j\in\set{1,2,\ldots, N}^*}}.\]
Let $K$ be a self-similar set in $\R$ generated by an IFS $\mathcal G$.
Fraser et al. \cite[Theorem 2.1]{FHOR2015} showed that if the IFS $\mathcal{G}$ satisfies the WSP and $K$ is not a singleton, then $K$ is Ahlfors regular.
Recall that a set $K \subset \mathbb{R}^d$ with Hausdorff dimension $s$ is called \emph{Ahlfors regular} if there exists a constant $c\ge1$ such that for all $0< r < 1$ and all $x\in K$,
$$ c^{-1} r^s \le \mathcal{H}^{s}\big( K\cap B(x,r) \big) \le cr^s, $$
where $\mathcal{H}^{s}$ is the $s$-dimensional Hausdorff measure, and $B(x,r)$ is a ball centered at $x$ with radius $r$.

Next we show that the IFS ${\mathcal F}_{q, A}$ defined in (\ref{eq:IFS-q-A}) satisfies the WSP. This can be deduced from Ngai and Wang \cite[Theorem 2.7]{Ngai-Wang-2001} and Nguyen \cite{Nguyen-2002}. For the reader's convenience  we give a direct proof.
The following equivalent condition of WSP was proved by Zerner \cite{Zerner1996}, {and we give a slightly differrent formulation in the book \cite{BSS2023}}.

\begin{proposition}[{{\cite[Theorem 4.2.4]{BSS2023}}}]
  \label{prop:equivalence-WSP}
  Let $K$ be a self-similar set in $\R$ generated by an IFS $\mathcal G$, which is not a singleton. Then the IFS $\mathcal G$ satisfies the WSP, if and only if for any $\rho >1$ there exist $x_0\in \R$ and $\epsilon>0$ such that
\begin{equation*}\label{equibalence-of-wsp}
	\forall~h\in {\mathcal{N_{\mathcal G}}} \quad\textrm{with}\quad r_h\in[\rho^{-1},\rho], \quad h(x_0)\ne x_0 \quad\Longrightarrow\quad |h(x_0)-x_0|>\epsilon,
\end{equation*}
where $r_h$ denotes the contraction/dilation coefficient of a similarity map $h$.
\end{proposition}

\begin{lemma}\label{IFS-is-WSP}
Let  $q\in \mathbb N_{\ge 3}$ and let $A\subset\mathbb Q$ be a finite set. Then  the IFS ${\mathcal F}_{q, A}$ defined in (\ref{eq:IFS-q-A}) satisfies the WSP.
\end{lemma}
\begin{proof}
Since $A\subset\mathbb Q$ is a finite set, we can choose $N\in \N$ such that $N a \in \Z$ for any $a \in A$.
Given $\rho>1$, let $x_0=0$ and  $\epsilon=(N\rho)^{-1}$. Let $h=f_{\mathbf i}^{-1}\circ f_{\mathbf j}$ for some $\mathbf{i}=i_1i_2\ldots i_{n}$, $\mathbf{j}=j_1j_2\ldots j_m \in A^{*}$. Observe that
\begin{align*}
f_{\mathbf{i}}^{-1}(x)&=f_{i_n}^{-1}\circ {f_{i_{n-1}}^{-1}}\circ\ldots\circ f_{i_1}^{-1}(x)=q^{n}x-\sum_{k=1}^{n} i_k q^{n-k},\\
f_{\mathbf{j}}(x)&=f_{j_1}\circ f_{j_2}\circ\ldots \circ f_{j_m}(x)=\frac{x}{q^m}+\sum_{k=1}^{m}\frac{{j_k}}{q^k}.
\end{align*}
Then
\begin{equation}\label{eq:formula-of-h}
	h(x)=f_{\mathbf i}^{-1}\circ f_{\mathbf j}(x)=q^{n-m}x+\sum_{k=1}^{m} {j_k} q^{n-k} -\sum_{k=1}^{n} {i_k} q^{n-k}.
\end{equation}

Suppose $r_h=q^{n-m}\in[\rho^{-1}, \rho]$ and $h(0)\ne 0$. If
 $n\ge m$, then by (\ref{eq:formula-of-h}) it follows that
\begin{equation*}
\begin{aligned}
	|N h(0)|=\left|\sum_{k=1}^{m} N {j_k} q^{n-k} -\sum_{k=1}^{n} N {i_k} q^{n-k} \right|\ge 1 \ge \rho^{-1},
\end{aligned}
\end{equation*}
where the first inequality follows since $h(0)\ne 0$ and $N a\in\mathbb Z$ for any $a\in A$.
If $n<m$, then by (\ref{eq:formula-of-h}) we obtain that
\begin{equation*}
\begin{aligned}
	|Nh(0)|&= q^{n-m} \left|\sum_{k=1}^{m} N {j_k} q^{m-k} -\sum_{k=1}^{n} N {i_k} q^{m-k}\right|\ge q^{n-m} \ge \rho^{-1}.
\end{aligned}
\end{equation*}
Thus we conclude that $|h(0)|\ge (N\rho)^{-1}=\epsilon$.  This completes the proof by Proposition \ref{prop:equivalence-WSP}.
\end{proof}

Recall from (\ref{eq:self-similar-set-K}) that each $x\in K(q, A)$ can be written as $x=\sum_{i=1}^{\f} a_i q^{-i}$ with each $a_i \in A$. The infinite sequence $(a_i)= a_1 a_2\ldots\in A^\N$ is called a \emph{coding} of $x$.
The following lemma is inspired by Schleischitz~\cite[Theorem 4.3]{Schleischitz-2021}.

\begin{lemma}\label{lem:order-upper-bound}
Let $q\in\N_{\ge 3}$ and let $A  \subset \mathbb Q$ be a finite set. If $s=\dim_{H}K(q, A) <1 $, then there exists a constant $c_1>0$ such that any rational number $\xi = v/u \in K(q, A)$ with $u \in \N$ and $\gcd(v,u)=1$ admits an eventually periodic coding with the periodic length $\le c_1 u^{s}.$
\end{lemma}
\begin{proof}
Take a rational number $\xi=v/u \in K(q, A)$ with $u \in \N$ and $\gcd(v,u)=1$.
Let $(a_i)=a_1a_2\ldots\in A^\N$ be a coding of $\xi$, i.e., $\xi=\sum_{i=1}^{\f}a_i q^{-i}$. Consider the sequence $\{\xi_n\}_{n=0}^\f$ defined by
\begin{equation}\label{eq:sequence-xi}
\xi_0:=\xi,\quad \textrm{and} \quad\xi_n:=f_{a_n}^{-1}(\xi_{n-1})= q \xi_{n-1} - a_n \quad\forall n\in \N,
\end{equation}
where $f_a(x) = (x+a)/q$ for $a \in A$.
Then we have $\xi_n = \sum_{i=1}^{\f} a_{n+i} q^{-i} \in K(q,A)$ for all $n\in \N_0$.

Since $A\subset\mathbb Q$ is a finite set, choose $N\in \N$ such that $N a \in \Z$ for any $a \in A$. 
Note that $\xi_0 \in \Z/u$ and $a_n \in \Z/N$ for all $n \in \N$.
By (\ref{eq:sequence-xi}), we can recursively show each $\xi_n$ has the form
$$\xi_n=\frac{v_n}{uN} \quad\textrm{for some } v_n \in \mathbb Z. $$
Let $\Upsilon:=\{\xi_n:n\in \N_0 \}$. Then for any two distinct elements $\eta_1, \eta_2 \in \Upsilon$, we have
\begin{equation}\label{eq:gap}
|\eta_1-\eta_2 |\ge \frac{1}{uN}.
\end{equation}

Note by Lemma \ref{IFS-is-WSP} that the IFS ${\mathcal F}_{q, A}$ satisfies the WSP. Then by Fraser et al. \cite[Theorem 2.1]{FHOR2015}, the self-similar set $K(q, A)$ is Ahlfors regular. Thus there exists a constant $c\ge1$ such that for all $0< r < 1$ and all $x\in K(q,A)$,
\begin{equation}\label{eq:Ahlfors-regular}
c^{-1} r^s \le \mathcal{H}^{s}\big( K(q,A)\cap B(x,r) \big) \le cr^s,
\end{equation}
where $s = \dim_H K(q,A) < 1$.
By (\ref{eq:gap}), the balls $\big\{ B(\eta, \frac{1}{3uN}): \eta \in \Upsilon\big\}$ are pairwise disjoint.
Note that $\Upsilon \subset K(q,A)$.
It follows from (\ref{eq:Ahlfors-regular}) that
\begin{align*}
	\mathcal{H}^{s}\big( K(q,A) \big) &\ge \mathcal{H}^{s}\bigg( K(q,A)\cap \bigcup_{{\eta\in \Upsilon}} B\Big( \eta,\frac{1}{3uN} \Big) \bigg)\\
	&=\sum_{\eta\in \Upsilon}\mathcal{H}^{s}\bigg( K(q,A)\cap B\Big( \eta,\frac{1}{3uN} \Big)\bigg)\\
	&\ge c^{-1}\Big(\frac{1}{3uN}\Big)^{s}\#\Upsilon.
\end{align*}
Since the set $K(q,A)$ is compact, it can be covered by finitely many balls with radius $1/2$. By (\ref{eq:Ahlfors-regular}) we also have $\mathcal{H}^{s}\big( K(q,A) \big)<+\f$. 
Thus there exists a constant $c_1>0$ depending only on ${\mathcal F}_{q,A}$ such that
$$\#\Upsilon \le c_1u^{s}.$$
This implies that there are $0\le j < k \le c_1 u^{s}$ such that $\xi_k = \xi_j$. By (\ref{eq:sequence-xi}) it yields that
$$\xi_j =f_{a_{j+1}}\circ {f_{a_{j+2}}} \circ \cdots \circ f_{a_{k}} (\xi_j).$$
This means $\xi_j$ has a periodic coding $(a_{j+1}a_{j+2}\ldots a_{k})^{\f}$.
It follows that $\xi$ has a eventually periodic coding $a_1a_2\ldots a_j (a_{j+1}a_{j+2}\ldots a_{k})^{\f}$ with the period length $k-j \le c_1 u^{s}.$
\end{proof}

For $a \in \N$ and $p\in \N_{\ge 2}$ with $\gcd(a,p)=1$, the \emph{order of $a$ modula $p$}, denoted by $\mathrm{ord}_p(a)$, is defined to be the smallest positive integer $n$ such that $a^{n} \equiv 1 \pmod{p}$.
By \cite[Theorem 88]{Hardy2010}, we have $a^{n} \equiv 1 \pmod{p}$ if and only if $\mathrm{ord}_p(a) \mid n$.
The following lemma was proven by Bloshchitsyn \cite{Bloshchitsyn-2015}.
\begin{lemma}[{\cite[Lemma 3]{Bloshchitsyn-2015}}]
  \label{lem:ord-a-b}
  Let $q\in\N_{\ge 3}$ and let $p$ be a  prime number with $\gcd(p,q)=1$. Then there exist  $m\in\N$ and $d\in\N$  such that
  \begin{equation*}\label{lemma-of-Bloshchitsyn}
\mathrm{ord}_{p}(q) = \mathrm{ord}_{p^2}(q) =\ldots=\mathrm{ord}_{p^{m}}(q)=d,\quad\textrm{and}\quad \mathrm{ord}_{p^{m+n}}(q)= p^{n} d \quad \forall n\in\N.
\end{equation*}
\end{lemma}

\begin{lemma}\label{lem:order-lower-bound}
Let $q \in \mathbb N_{\ge 3}$ and let $p_1, p_2, \ldots, p_k$ be different   prime numbers such that $\gcd(p_i,q)=1$ for all $1\le i\le k$. Then there exists a constant $c_2>0$ such that
\[ \mathrm{ord}_{p_1^{n_1}p_2^{n_2}\ldots p_k^{n_k}}(q)\ge c_2 p_1^{n_1}p_2^{n_2}\ldots p_k^{n_k} \]
for all $(n_1, n_2, \ldots, n_k)\in\N_0^k\setminus\{\mathbf{0}\}$.
\end{lemma}
\begin{proof}
For each $1\le i \le k$, by Lemma \ref{lem:ord-a-b} there exist $m_i\in\N$ and $d_i\in\N$ such that
\[ \mathrm{ord}_{p_i}(q) = \mathrm{ord}_{p_i^2}(q) =\ldots=\mathrm{ord}_{p_i^{m_i}}(q)=d_i,\quad\text{and}\quad \mathrm{ord}_{p_i^{m_i+n}}(q)=p_i^{n} d_i \quad \forall n \in \N,\]
and thus,
\begin{equation}\label{eq:p-i}
  p_i^n ~\mid~ p_i^{m_i} \cdot \mathrm{ord}_{p_i^{n}}(q) \quad \forall n \in \N.
\end{equation}

Fix any $(n_1, n_2, \ldots, n_k)\in\N_0^k\setminus\{\mathbf{0}\}$.
Note by \cite[Theorem 88]{Hardy2010} that for each $1 \le i \le k$ we have
$$\mathrm{ord}_{p_i^{n_i}}(q) ~\mid~\mathrm{ord}_{p_1^{n_1}p_2^{n_2}\ldots p_k^{n_k}}(q).$$
For convenience, write $\mathrm{ord}_{1}(q)=1$.
By (\ref{eq:p-i}), we have $$p_i^{n_i}~\mid~ p_i^{m_i} \cdot \mathrm{ord}_{p_1^{n_1}p_2^{n_2}\ldots p_k^{n_k}}(q)\quad \forall 1 \le i \le k.$$
It follows that $$p_1^{n_1}p_2^{n_2}\ldots p_k^{n_k}~\mid~ p_1^{m_1} p_2^{m_2} \ldots p_k^{m_k} \cdot \mathrm{ord}_{p_1^{n_1}p_2^{n_2}\ldots p_k^{n_k}}(q).$$
This implies that $$\mathrm{ord}_{p_1^{n_1}p_2^{n_2}\ldots p_k^{n_k}}(q) \ge \frac{1}{p_1^{m_1} p_2^{m_2} \ldots p_k^{m_k}}p_1^{n_1}p_2^{n_2}\ldots p_k^{n_k}, $$
as desired.
\end{proof}

Based on Lemmas \ref{lem:order-upper-bound} and  \ref{lem:order-lower-bound}, we prove Theorem \ref{thm:equivalence-of-finite} now.

\begin{proof}[Proof of Theorem \ref{thm:equivalence-of-finite}]
It remains to show (i) implies (iii). Assume $s=\dim_H K(q,A) < 1$.

Let $p\in\N_{\ge 2}$ with $\gcd(p,q)=1$.
Recall from (\ref{Dp-definition}) that $D_p$ consists of all rational numbers with a finite $p$-ary expansions.
Let $p_1,p_2,\ldots,p_k$ be all distinct prime factors of $p$.
Write $\mathbf{p}=(p_1,p_2,\ldots,p_k)$. For $\mathbf{n}=(n_1,n_2,\ldots,n_k)\in\N_0^k$ we define
$$D_{\mathbf{p}}^{\mathbf{n}}:=\bigg\{ \frac{d}{p_1^{n_1}p_2^{n_2}\ldots p_{k}^{n_k} }:d\in \mathbb {Z}, ~\gcd(d,p_1^{n_1}p_2^{n_2}\ldots p_{k}^{n_k})=1 \bigg\}.$$
In particular, $D_{\mathbf{p}}^{\mathbf{0}} :=\mathbb Z$. Then $D_p$ can be written as disjoint unions of $D_{\mathbf{p}}^{\mathbf{n}}$, i.e.,
\[ D_p =\bigcup_{\mathbf{n} \in\N_0^k} D_{\mathbf{p}}^{\mathbf{n}}. \]
Note that each $D_{\mathbf{p}}^{\mathbf{n}}$ is uniformly discrete.
We have $\# \big( D_{\mathbf{p}}^{\mathbf{n}} \cap K(q,A) \big) < +\f$ for each $\mathbf{n}\in\N_0^k$.
In order to show $\# \big( D_p \cap K(q,A) \big) < +\f$, it suffices to show that there exists $h \in \N$ such that
\begin{equation}\label{eq:h-1}
  D_{\mathbf{p}}^{\mathbf{n}} \cap K(q,A) = \emptyset \quad \forall\mathbf{n}=(n_1,\ldots,n_k)\in\N_0^k \text{ with } n_1 + \ldots + n_k > h.
\end{equation}

Let $c_1,c_2 > 0$ be two constants as in Lemma \ref{lem:order-upper-bound} and Lemma \ref{lem:order-lower-bound}, respectively.
Since $A\subset\mathbb Q$ is a finite set, let $N\in \N$ satisfy $N a \in \Z$ for each $a \in A$.
Choose a large enough integer $h \in \N$ such that
\begin{equation}\label{eq:h-2}
  2^{h(1-s)} > \frac{c_1 N}{c_2}.
\end{equation}
In the following, we show the integer $h$ defined in (\ref{eq:h-2}) satisfies (\ref{eq:h-1}) by contradiction.

Suppose that there exists $\mathbf{n}=(n_1,n_2,\ldots,n_k)\in\N_0^k$ with $n_1 + n_2 + \ldots + n_k > h$ such that $D_{\mathbf{p}}^{\mathbf{n}} \cap K(q,A) \ne \emptyset$.
Take $$x=\frac{d}{p_1^{n_1}p_2^{n_2}\ldots p_{k}^{n_k}} \in D_{\mathbf{p}}^{\mathbf{n}} \cap K(q,A),$$ where $d\in\mathbb Z$ and $\gcd(d, p_1^{n_1}p_2^{n_2}\ldots p^{n_k}_k)=1$.
By Lemma \ref{lem:order-upper-bound}, the rational number $x$ has an eventually periodic coding  $a_1a_2\ldots a_j(a_{j+1} a_{j+2} \ldots a_{j+n})^\f\in A^\N$ with
 \begin{equation}\label{eq:upper-bound-k}
   n\le c_1(p_1^{n_1}p_2^{n_2}\ldots p_k^{n_k})^s.
 \end{equation}

Note that
\begin{align*}
\frac{d}{p_1^{n_1}p_2^{n_2}\ldots p_{k}^{n_k}}=x&=\sum_{i=1}^{j} \frac{a_i}{q^i} +\left(1+\frac{1}{q^{n}} +\frac{1}{q^{2n}}+\ldots  \right)\sum_{i=j+1}^{j+n} \frac{a_i}{q^i}\\
&=\sum_{i=1}^{j} \frac{a_i}{q^i} +\frac{q^n}{q^n-1 }\sum_{i=j+1}^{j+n} \frac{a_i}{q^i},
\end{align*}
and all $a_i \in \Z/N$.
Then we have
\[ \frac{d}{p_1^{n_1}p_2^{n_2}\ldots p_{k}^{n_k}} \in \frac{\mathbb Z}{Nq^j(q^{n}-1)}.\]
Note that $\gcd(d,p_1^{n_1}p_2^{n_2}\ldots p_k^{n_k})=1$ and $\gcd(p,q)=1$. It follows  that \[ p_1^{n_1}p_2^{n_2}\ldots p_k^{n_k} \mid N(q^n -1).\]
Write $\gcd(N,p_1^{n_1}p_2^{n_2}\ldots p_k^{n_k})=p_1^{m_1}p_2^{m_2}\ldots p_k^{m_k}$ with $(m_1,m_2,\ldots,m_k) \in \N_0^k$.
Then we obtain \[ p_1^{n_1-m_1}p_2^{n_2-m_2}\ldots p_k^{n_k-m_k} ~\mid~ q^n -1, \]
that is, \[ q^n \equiv 1 \pmod{p_1^{n_1-m_1}p_2^{n_2-m_2}\ldots p_k^{n_k-m_k}}. \]
By \cite[Theorem 88]{Hardy2010}, we have $$\mathrm{ord}_{p_1^{n_1-m_1}p_2^{n_2-m_2}\ldots p_k^{n_k-m_k}}(q) ~\mid~ n. $$
By Lemma \ref{lem:order-lower-bound}, we get
\begin{equation*}
  n\ge \mathrm{ord}_{p_1^{n_1-m_1}p_2^{n_2-m_2}\ldots p_k^{n_k-m_k}}(q) \geq c_2 p_1^{n_1-m_1}p_2^{n_2-m_2}\ldots p_k^{n_k-m_k}\ge \frac{c_2}{N} p_1^{n_1} p_2^{n_2} \ldots p_k^{n_k},
\end{equation*}
where the last inequality follows from $p_1^{m_1}p_2^{m_2}\ldots p_k^{m_k} \mid N$.
Together with (\ref{eq:upper-bound-k}), we conclude that $$(p_1^{n_1} p_2^{n_2} \ldots p_k^{n_k})^{1-s} \le \frac{c_1 N}{c_2}.$$
Note that all $p_j \ge 2$ and $n_1 +n_2 + \ldots + n_k>h$.
Thus we have $$2^{h(1-s)} < (p_1^{n_1} p_2^{n_2} \ldots p_k^{n_k})^{1-s} \le \frac{c_1 N}{c_2},$$
which contradicts (\ref{eq:h-2}).
This completes the proof.
\end{proof}

At the end of this section, we apply Theorem \ref{thm:equivalence-of-finite} to prove Theorem \ref{thm:uniform-bound}.

\begin{proof}[Proof of Theorem \ref{thm:uniform-bound}]
  For $n \in \N$, let $Q_n$ be the set of all rational numbers with denominator $\le n$.
  Since $\#(A-A) < q$, we clearly have $$\dim_H K(q,A-A) \le \frac{\log \#(A-A)}{\log q} < 1.$$
  By Theorem \ref{thm:equivalence-of-finite}, we have $\#\big( D_p \cap K(q,A-A) \big) < +\f$.
  Thus, there exists $N \in \N$ such that $$D_p \cap K(q,A-A) \subset Q_N.$$
  We claim that for any $\alpha \in \R$, $$\# \big( (D_p + \alpha) \cap K(q,A) \big) \le N \cdot \mathrm{diam}~ K(q,A) + 1. $$

  Suppose on the contrary that there exists $\alpha_0 \in \R$, $$\# \big( (D_p + \alpha_0) \cap K(q,A) \big) > N \cdot \mathrm{diam}~ K(q,A) + 1. $$
  Then we can find $x_1, x_2, \ldots, x_k \in D_p$ {with $k > N \cdot \mathrm{diam}~ K(q,A) + 1$} such that $$x_i + \alpha_0 \in K(q,A) \quad \forall 1 \le i \le k.$$
  We can assume that $x_1 < x_2 < \ldots < x_k$.
  For any $1 \le i < k$, we have
  $$x_{i+1}-x_i = (x_{i+1}+\alpha_0) - (x_i + \alpha_0) \in K(q,A) - K(q,A) = K(q,A-A),$$
  and in the meantime, $$x_{i+1}-x_i \in D_p - D_p = D_p.$$
  Thus we obtain $$x_{i+1}-x_i \in  D_p \cap K(q,A-A) \subset Q_N \quad \forall 1 \le i < k.$$
  This implies that $$x_{i+1}-x_i \ge \frac{1}{N} \quad \forall 1 \le i < k. $$
  Note that $k > N \cdot \mathrm{diam}~ K(q,A) + 1$. It follows that
  \begin{align*}
    \mathrm{diam}~ K(q,A)  \ge (x_k + \alpha_0) - (x_1+\alpha_0)
     = \sum_{i=1}^{k -1} (x_{i+1} - x_i)
     \ge \frac{k-1}{N}
     > \mathrm{diam}~ K(q,A),
  \end{align*}
  a contradiction.
  This completes the proof.
\end{proof}

\section{Spectral eigenvalues of self-similar spectral measures}\label{sec:spectral-eigenvalue}

In this section we study the spectral eigenvalue problem.
By applying Theorem \ref{thm:equivalence-of-finite}, we give a class of spectral eigenvalues of self-similar spectral measures, and prove Theorem \ref{thm:spectral-eigenvalue}.

For a Borel probability measure $\mu$ on $\R$, the Fourier transform is defined by
\[ \widehat{\mu}(\xi) = \int_\R e^{-2\pi \mathrm{i} \xi x} \mathrm{d}\mu(x),\; \xi \in \R. \]
The following criterion is frequently used for spectral measures. 
\begin{theorem}[{\cite[Theorem 2.6]{Li-Miao-Wang-2024}}]
  \label{thm:spectral-measure}
  Let $\mu$ be a Borel probability measure on $\R$ and $\Lambda \subset \R$.
  Then the set $\Lambda$ is a spectrum of $\mu$ if and only if \[\sum_{\lambda \in \Lambda} \big| \widehat{\mu}(\xi+\lambda) \big|^2 = 1 \quad \forall \xi \in \R.\]
\end{theorem}

Using the criterion in Theorem \ref{thm:spectral-measure}, one can easily check the following lemma.
\begin{lemma}\label{lemma:spectrum}
Let $\Lambda$ be a spectrum of $\mu$. Then we have
\begin{itemize}
  \item[{\rm(i)}] $\Lambda+t$ is a spectrum of $\mu$ for any $t \in \R$;
  \item[{\rm(ii)}] $-\Lambda$ is a spectrum of $\mu$;
  \item[{\rm(iii)}] $\Lambda$ is a spectrum of $\mu\circ T_a$, where $T_a(x) = x + a$ and $a \in \R$.
\end{itemize}
\end{lemma}
By Lemma \ref{lemma:spectrum} (i) and (ii), we always assume that $0 \in \Lambda$, and consider only positive spectral eigenvalues.

For a finite set $F \subset \R$, define the discrete measure $\delta_F := \frac{1}{\#F} \sum_{a \in F} \delta_a$, where $\delta_a$ is the Dirac measure.
Let $N \in \N_{\ge 2}$ and let $B,L \subset \Z $ be finite sets with $\# B = \#L \ge 2$.
For $\ell\ne \ell' \in L$, we have \[ \big\langle e_{\ell}(x), e_{\ell'}(x) \big\rangle_{L^2(\delta_{N^{-1} B})} = \int_{\R} e^{2\pi \mathrm{i} (\ell-\ell')x} \mathrm{d} \delta_{N^{-1} B}(x) = \frac{1}{\#B}\sum_{b \in B} e^{2\pi \mathrm{i} \frac{b(\ell-\ell')}{N}}. \]
The orthogonality of matrix in (\ref{eq:matrix}) is equivalent to that $\{e_\ell(x): \ell \in L\}$ is orthogonal in $L^2(\delta_{N^{-1} B})$.
Note that the dimension of $L^2(\delta_{N^{-1} B})$ is equal to $\#B = \#\{e_\ell(x): \ell \in L\}$.
Thus, we conclude that $(N,B,L)$ is a Hadamard triple if and only if $L$ is a spectrum of $\delta_{N^{-1} B}$.
By Lemma \ref{lemma:spectrum} (i) and (iii), we immediately obtain the following lemma.
\begin{lemma}\label{lemma:Hadamard-triple}
  Let $(N,B,L)$ be a Hadamard triple in $\R$.
  Then $(N,B+b_0,L+\ell_0)$ is also a Hadamard triple for any $b_0,\ell_0 \in \Z$.
\end{lemma}

The following lemma can be viewed to {provide} a class of spectral eigenvalues of discrete spectral measure.

\begin{lemma}\label{lemma:multiply-p}
  Let $(N,B,L)$ be a Hadamard triple in $\R$.
  If $p \in \N_{\ge 2}$ with $\gcd(p,N)=1$, then $(N,B,pL)$ is also a Hadamard triple.
\end{lemma}
\begin{proof}
  By Lemma \ref{lemma:Hadamard-triple}, we can assume that $B \subset \N$.
  It suffices to show that $\{e_{p\ell}(x): \ell \in L\}$ is orthogonal in $L^2(\delta_{N^{-1} B})$, i.e.,
  \[ \sum_{b \in B} e^{2\pi \mathrm{i} \frac{pb(\ell-\ell')}{N}} =0 \quad \forall \ell\ne \ell' \in L.\]

  Fix $\ell,\ell' \in L$ with $\ell> \ell'$. Write $P(x) = \sum_{b \in B} x^{b(\ell-\ell')}$.
  Since $(N,B,L)$ is a Hadamard triple, we have $$P\big( e^{\frac{2\pi \mathrm{i}}{N}}  \big) = \sum_{b \in B} e^{2\pi \mathrm{i} \frac{b(\ell-\ell')}{N}} =0. $$
  Let $\Phi_N(x)$ be the $N$-th cyclotomic polynomial, i.e., the monic minimal polynomial of $e^{\frac{2\pi \mathrm{i}}{N}}$. Since $P(x) \in \Z[x]$, we obtain $\Phi_N(x) \mid P(x)$ in $\Z[x]$.
  Note that $$\Phi_N(x) = \prod_{1\le k \le N, \gcd(k,N)=1} \big( x - e^\frac{2\pi k \mathrm{i}}{N} \big).$$
  Since $\gcd(p,N)=1$, we have $\Phi_N\big( e^{\frac{2\pi p\mathrm{i} }{N}} \big)=0$.
  It follows that $$P\big( e^{\frac{2\pi \mathrm{i} p}{N}} \big)= \sum_{b \in B} e^{2\pi \mathrm{i} \frac{pb(\ell-\ell')}{N}} =0,$$
  as desired.
\end{proof}

Let $(N,B,L)$ be a Hadamard triple in $\R$ with $0 \in B \cap L$.
Define $$m_B(x) := \frac{1}{\# B} \sum_{b \in B} e^{2\pi \mathrm{i} b x}.$$
A finite set $C=\{x_1,x_2,\ldots, x_k\}$ is called a \emph{cycle} for the IFS $\mathcal{F}_{N,L}=\big\{ \tau_\ell(x) = (x+\ell)/N: \ell \in L \big\}$ if there exist $\ell_1,\ell_2,\ldots \ell_k \in L$ such that \[ \tau_{\ell_i}(x_i) = x_{i+1} \text{ for } i=1,2,\ldots, k-1, \text{ and } \tau_{\ell_k}(x_k) = x_1.  \]
The cycle $C$ is called an \emph{$m_B$-cycle} if $|m_B(x_i)|=1$ for all $1\le i \le k$.
Note that $\{0\}$ is a trivial $m_B$-cycle.
{Dutkay and Jorgensen \cite{Dutkay-Jorgensen-2008} gave the following construction of spectrum for self-similar spectral measures. }

\begin{theorem}[{\cite[Theorem 2.38]{Dutkay-Jorgensen-2008}}]
  \label{thm:spectrum}
  Let $(N,B,L)$ be a Hadamard triple in $\R$ with $0\in B \cap L$.
  Let $\Lambda$ be the smallest set that contains $-C$ for all $m_B$-cycles $C$, and which satisfies $N \Lambda + L \subset \Lambda$.
  Then the set $\Lambda$ is a spectrum of the self-similar measure $\mu_{N,B}$.
\end{theorem}

The spectrum $\Lambda$ in Theorem \ref{thm:spectrum} can be constructed explicitly.
As before, let $K(N,L)$ be the self-similar set generated by the IFS $\mathcal{F}_{N,L}=\big\{ \tau_\ell(x) = (x+\ell)/N: \ell \in L \big\}$.
\begin{proposition}\label{prop:construct-spectrum}
  Let $(N,B,L)$ be a Hadamard triple in $\R$ with $0\in B \cap L$ and write $d = \gcd(B)$.
  Let $\Lambda_0= -\big( K(N,L) \cap (\Z/d) \big)$ and $\Lambda_n = N \Lambda_{n-1} + L$ for $n \in \N$.
  Then $\Lambda_{n-1} \subset \Lambda_n$ for all $n \in \N$, and the set $$\Lambda = \bigcup_{n=0}^\f \Lambda_n$$
  is a spectrum of the self-similar measure $\mu_{N,B}$.
\end{proposition}
\begin{proof}
  We first show that any $m_B$-cycle is contained in $K(N,L) \cap (\Z/d)$.
  Let $C=\{x_1,x_2,\ldots, x_k\}$ be an $m_B$-cycle. Then there exist $\ell_1,\ell_2,\ldots \ell_k \in L$ such that \[ \tau_{\ell_i}(x_i) = x_{i+1} \text{ for } i=1,2,\ldots, k-1, \text{ and } \tau_{\ell_k}(x_k) = x_1.  \]
  It follows that $x_1 = \tau_{\ell_k} \circ \tau_{\ell_{k-1}} \circ \cdots \circ \tau_{\ell_1}(x_1)$, i.e., $x_1$ is the fixed point of $\tau_{\ell_k} \circ \tau_{\ell_{k-1}} \circ \cdots \circ \tau_{\ell_1}$.
  Thus we obtain $x_1 \in K(N,L)$. Note that $x_i = \tau_{\ell_{i-1}} \circ \cdots \circ \tau_{\ell_1}(x_1)$ for all $2 \le i \le k$. It follows that $ C \subset K(N,L)$.
  On the other hand, for each $1 \le i \le k$, we have \[ \big|m_B(x_i)\big| = \bigg| \frac{1}{\#B} \sum_{b \in B} e^{2\pi \mathrm{i} b x_i} \bigg| =1. \]
  Note that $0 \in B$. For each $1 \le i \le k$, we must have $bx_i \in \Z$ for any $b \in B$.
  Since $\gcd(B) = d$, it follows that $x_i \in \Z/ d$ for all $1 \le i \le k$.
  Therefore, we conclude that $C \subset K(N,L) \cap (\Z/d)$.

  Next, we show that any point in $K(N,L) \cap (\Z/d)$ is contained in some $m_B$-cycle.
  Take any $x_0 \in K(N,L) \cap (\Z/d)$.
  Since $x_0 \in K(N,L)$, for each $i \in \N$, we can recursively find $\ell_i \in L$ and $x_{i} \in K(N,L)$ such that $x_{i-1} = \tau_{\ell_i}(x_i)$.
  Note that $x_0 \in \Z/d$ and $x_i = N x_{i-1} - \ell_i$. We also have $x_i \in \Z / d$ for all $i \in \N$.
  Thus, we obtain $\{x_i\}_{i=0}^\f  \subset K(N,L) \cap (\Z/d)$.
  Note that $K(N,L) \cap (\Z/d)$ is a finite set.
  Let $k = \min\{ i \in \N: \text{ there exists $0\le i' < i$ such that $x_i = x_{i'}$} \}$.
  We claim $x_k = x_0$.
  Otherwise, there exists $1 \le  j < k$ such that $x_k = x_j$.
  Then we have $x_{k-1} - x_{j-1} = \tau_{\ell_k}(x_k) - \tau_{\ell_j}(x_j) = (\ell_k - \ell_j)/N$.
  The minimality of $k$ implies $x_{k-1} \ne x_{j-1}$, i.e., $\ell_k \ne \ell_j$.
  Note that $b(\ell_k - \ell_j)/N = b (x_{k-1} - x_{j-1}) \in \Z$ for any $b \in B$, since $\gcd(B) = d$ and $x_i \in \Z/d$ for all $i \in \N_0$.
  It follows that \[ \frac{1}{\#B}\sum_{b \in B} e^{2\pi \mathrm{i} \frac{b(\ell_k -\ell_j)}{N} } = 1, \;\; \ell_k \ne \ell_j \in L. \]
  This contradicts that $(N,B,L)$ is a Hadamard triple.
  Thus, we conclude that $x_k = x_0$.
  It follows that $C =\{x_k=x_0, x_{k-1}, \ldots, x_1\}$ is cycle.
  Note that $C \subset \Z / d$ and $\gcd(B) = d$.
  It is easy to check that $m_B(x) = 1$ for any $x \in C$.
  Thus, $C$ is an $m_B$-cycle and $x_0 \in C$.

  We have showed that $\Lambda_0$ is the union of $-C$ for all $m_B$-cycles $C$.
  For any $x_0 \in \Lambda_0$, we have $-x_0$ is in some $m_B$-cycle. Then there exists $\ell_0 \in L$ such that $x_1=\tau_{\ell_0}(-x_0)$ is also in some $m_B$-cycle.
  Then we obtain $x_0 = N(-x_1) + \ell_0 \in N \Lambda_0 + L = \Lambda_1$.
  Thus, we get $\Lambda_0 \subset \Lambda_1$.
  It follows that $\Lambda_1 = N \Lambda_0 + L \subset N \Lambda_1 + L = \Lambda_2$.
  Therefore, we can recursively show that $\Lambda_{n-1} \subset \Lambda_n$ for all $n \in \N$.

  By the construction of $\Lambda$, we obviously have $\Lambda_0 \subset \Lambda$ and $N \Lambda + L \subset \Lambda$.
  Suppose that $\Lambda' \subset \R$ satisfies $\Lambda_0 \subset \Lambda'$ and $N \Lambda' + L \subset \Lambda'$.
  Then $\Lambda_1 = N \Lambda_0 + L \subset N \Lambda' + L \subset \Lambda'$.
  It can be showed by recursion that $\Lambda_n \subset \Lambda'$ for all $n \in \N_0$. It yields that $\Lambda \subset \Lambda'$.
  Therefore, we conclude that $\Lambda$ is the smallest set satisfying $\Lambda_0 \subset \Lambda$ and $N \Lambda + L \subset \Lambda$.
  We complete the proof by Theorem \ref{thm:spectrum}.
\end{proof}

Based on the construction of spectrum in Proposition \ref{prop:construct-spectrum}, we apply Theorem \ref{thm:equivalence-of-finite} to prove Theorem \ref{thm:spectral-eigenvalue}.

\begin{proof}[Proof of Theorem \ref{thm:spectral-eigenvalue}]
  Let $(N,B,L)$ be a Hadamard triple in $\R$ with $\# B < N$.
  By Lemma \ref{lemma:spectrum} (iii) and Lemma \ref{lemma:Hadamard-triple}, we can assume that $0 \in B \cap L$. Write $d = \gcd(B)$.
  Let $p_1,p_2,\ldots,p_k \in \N_{\ge 2}$ with $\gcd(p_1p_2\ldots p_k, N) =1$.
  Write $p = p_1 p_2 \ldots p_k$.
  Since $\#L= \#B < N$, we have $\dim_H K(N,L) < 1$.
  Note that $\gcd(p,N) =1$. By Theorem \ref{thm:equivalence-of-finite} and Remark \ref{remark:intersection} (i), we have $$\#\Big( K(N,L) \cap \frac{D_p}{d}  \Big)< +\f.$$
  Note that $D_p = \bigcup_{n=1}^\f \Z/p^n$ and the sequence $\{\Z/p^n\}_{n=1}^\f$ is increasing.
  Thus, there exists $n_0 \in \N$ such that
  \begin{equation}\label{eq:finite-n-0}
    K(N,L) \cap \frac{D_p}{d} = K(N,L) \cap \frac{\Z}{d p^{n_0}}.
  \end{equation}

  Next, we construct the set $\Lambda$.
  By Lemma \ref{lemma:multiply-p}, $(N,B,p^{n_0}L)$ is a Hadamard triple.
  Let $$\Lambda_0 = - \Big( K(N,p^{n_0}L) \cap \frac{\Z}{d}  \Big) = -p^{n_0}\Big( K(N,L) \cap \frac{\Z}{d p^{n_0}}  \Big),$$
  and $\Lambda_n = N \Lambda_{n-1} + p^{n_0}L$ for $n \in \N$.
  By Proposition \ref{prop:construct-spectrum}, the set $$\Lambda = \bigcup_{n=0}^\f \Lambda_n$$
  is a spectrum of $\mu_{N,B}$.

  Fix any tuple $(n_1,n_2,\ldots, n_k) \in \N_0^k$.
  Note that $\gcd(p_1^{n_1} p_2^{n_2} \ldots p_k^{n_k} p^{n_0}, N) =1$.
  By Lemma \ref{lemma:multiply-p}, $(N,B,p_1^{n_1} p_1^{n_2}\ldots p_k^{n_k}p^{n_0}L)$ is a Hadamard triple.
  Let
  \begin{align*}
    \Lambda'_0 & = - \Big( K(N,p_1^{n_1} p_2^{n_2} \ldots p_k^{n_k}p^{n_0}L) \cap \frac{\Z}{d} \Big) \\
    & = -p_1^{n_1} p_2^{n_2}  \ldots p_k^{n_k}p^{n_0}\Big( K(N,L) \cap \frac{\Z}{d p_1^{n_1} p_2^{n_2} \ldots p_k^{n_k}p^{n_0}} \Big),
  \end{align*}
  and $\Lambda'_n = N \Lambda'_{n-1} + p_1^{n_1} p_2^{n_2} \ldots p_k^{n_k} p^{n_0} L$ for $n \in \N$.
  By Proposition \ref{prop:construct-spectrum}, the set $$\Lambda' = \bigcup_{n=0}^\f \Lambda'_n$$
  is a spectrum of $\mu_{N,B}$.
  It remains to show $\Lambda' = p_1^{n_1} p_2^{n_2} \ldots p_k^{n_k} \Lambda$.

  Note that $$ \frac{\Z}{d p^{n_0}} \subset \frac{\Z}{d p_1^{n_1} p_2^{n_2} \ldots p_k^{n_k}p^{n_0}} \subset \frac{D_p}{d}.$$
  It follows from (\ref{eq:finite-n-0}) that $$K(N,L) \cap \frac{\Z}{d p_1^{n_1} p_2^{n_2} \ldots p_k^{n_k}p^{n_0}} = K(N,L) \cap \frac{\Z}{d p^{n_0}},$$
  i.e., $\Lambda'_0 = p_1^{n_1} p_2^{n_2} \ldots p_k^{n_k} \Lambda_0$.
  It follows that $$\Lambda'_1=N \Lambda'_{0} + p_1^{n_1} p_2^{n_2} \ldots p_k^{n_k} p^{n_0} L = p_1^{n_1} p_2^{n_2} \ldots p_k^{n_k} \big( N\Lambda_0 + p^{n_0} L \big) = p_1^{n_1} p_2^{n_2} \ldots p_k^{n_k} \Lambda_1.$$
  By recursion, we obtain $\Lambda'_n = p_1^{n_1} p_2^{n_2} \ldots p_k^{n_k} \Lambda_n$ for all $n \in \N_0$.
  It follows that $\Lambda' = p_1^{n_1} p_2^{n_2} \ldots p_k^{n_k} \Lambda$.
\end{proof}

\section*{Acknowledgements}
The first author was supported by Chongqing NSF  No.~CQYC20220511052.
The third author was supported by NSFC No.~12471085 and the China Postdoctoral Science Foundation No.~2024M763857.


\end{document}